\documentclass[a4paper,12pt]{amsart}
\usepackage{amssymb}
\usepackage{bbm}
\usepackage[dvips]{graphicx}
\usepackage{pdfsync}
\usepackage{color,xcolor}
\usepackage{verbatim}
\usepackage{enumerate}

\newtheorem{thm}{Theorem}[section]
\newtheorem*{thm*}{Theorem}

\newtheorem*{lem*}{Lemma}

\newtheorem*{cor*}{Corollary}
\newtheorem{prop}[thm]{Proposition}
\newtheorem*{prop*}{Proposition}
\theoremstyle{definition}
\newtheorem{defn}[thm]{Definition}
\newtheorem*{defn*}{Definition}
\theoremstyle{remark}

\newtheorem*{rem*}{Remark}

\newtheorem*{example*}{Example}

\newtheorem*{que*}{Question}

\newcommand{\abs}[1]{\left\vert#1\right\vert}
\newcommand{\set}[1]{\left\{#1\right\}}
\newcommand{\eps}{\varepsilon}

\newcommand{\CN}{\mathcal{N}}
\newcommand{\CT}{\mathcal{T}}

\newcommand{\CS}{\mathcal{S}}

\renewcommand{\emptyset}{\varnothing}
\renewcommand{\tilde}{\widetilde}

\newcommand{\sd}{\bigtriangleup}

\usepackage{mathrsfs}

\renewcommand{\epsilon}{\varepsilon}

\renewcommand{\leq}{\leqslant}
\renewcommand{\geq}{\geqslant}

\title{Extensions of full shifts with group actions}
\author{Bartosz Frej, Dawid Huczek}

\address{\hskip- \parindent
Bartosz Frej, Faculty of Pure and Applied Mathematics, Wroclaw University of Technology, Wybrze\.ze Wyspia\'nskiego 27, 50-370 Wroc\l aw, Poland}
\email{bartosz.frej@pwr.edu.pl}

\address{\hskip- \parindent
Dawid Huczek, Faculty of Pure and Applied Mathematics, Wroclaw University of Technology, Wybrze\.ze Wyspia\'nskiego 27, 50-370 Wroc\l aw, Poland}
\email{dawid.huczek@pwr.edu.pl}

\subjclass[2010]{Primary 37B10; Secondary 37B40}    

\keywords{countable amenable group, (dynamical) tiling, free action, topological entropy}

\date{\today}

\begin{document}

\begin{abstract}
We give a sufficient condition for a symbolic topological dynamical system with action of a countable amenable group to be an extension of the full shift, a problem analogous to those studied by Ashley, Marcus, Johnson and others for actions of $\mathbb{Z}$ and $\mathbb{Z}^d$.
\end{abstract}
\maketitle

\section{Introduction}
A well-known result in the study of symbolic dynamical systems states that any 
subshift of finite type (SFT) with the action of $\mathbb{Z}$ and entropy 
greater or equal than $\log n$ factors onto the full shift over $n$ symbols -- 
this was proven in \cite{M} and \cite{B} for the cases of equal and unequal 
entropy respectively. Extending these results for actions of other groups has 
been difficult, and it is known that a factor map onto a full shift of equal 
entropy may not exist in this case (see \cite{BS}). Johnson and Madden showed 
in \cite{JM} that any SFT with the action of $\mathbb{Z}^d$, which has entropy greater 
than $\log n$ and satisfies an additional mixing condition (known as corner gluing), 
has an extension which is finite-to-one (hence of equal entropy) 
and maps onto the full shift over $n$ symbols. This result was later improved 
by Desai in \cite{D} to show that such a system factors directly onto the full 
shift, without the intermediate extension.

In this paper we use similar methods to show that in the case of amenable group 
actions, any symbolic dynamical system with entropy greater than $\log n$ which 
satisfies a mixing condition (the \emph{gluing property}, see definition \ref{def:gluing}), has an 
equal-entropy symbolic extension which factors onto the full shift over $n$ 
symbols.

\section{Amenable groups and invariance}
\begin{defn}
	A countable group $G$ is called amenable if there exists a sequence $(F_n)$ 
	of finite subsets of $G$ (known as F\o lner sets) such that for every $g\in 
	G$ we have 
	\[ \lim\limits_{n\to\infty}\frac{\abs{gF_n\sd F_n}}{\abs{F_n}}=0, \]
	where $\abs{\cdot}$ denotes the cardinality of a set, and $\sd$ denotes the 
	symmetric difference. 
\end{defn}
\begin{defn}
	For a pair of finite sets $T,D\subset G$ and $\delta>0$, we say that $T$ is 
	$(D,\delta)$-invariant, if $\frac{\abs{DT\sd T}}{\abs{T}}<\delta$.
\end{defn}
Note that if $D$ contains the neutral element of $G$, then the above condition 
simplifies to $\frac{\abs{DT\setminus T}}{\abs{T}}<\delta$.
\begin{defn}
	If $T$ and $D$ are two finite subsets of $G$, the \emph{$D$-core of 
	$T$} is the set
	\[T_D=\set{t\in T:Dt\subset T}.\] 
\end{defn}
It is easy to check that for every $\eps>0$ there exists a $\delta>0$ such that 
if $T$ is $(D,\delta)$ invariant, then $\abs{T\setminus T_D}<\eps\abs{T}$, i.e., 
$T_D$ is a relatively large subset of $T$.
%
\section{Symbolic dynamical systems and entropy}
Let $\Lambda$ be any finite set. The \emph{full $G$-shift over $\Lambda$} 
(often referred to as just the full shift over $\Lambda$) is the dynamical 
system 
with the space $\Lambda^G$ endowed with the product topology, and the action of 
$G$ defined as $(gx)(h)=x(hg)$. A \emph{symbolic dynamical system over 
$\Lambda$} is any 
closed subset of $\Lambda^G$ invariant under the action of $G$. If $T\subset G$ 
is a finite set, then a \emph{block with domain $T$} is a mapping 
$B:T\to\Lambda$. If $x\in \Lambda^G$ and $T\subset G$ is a finite set, then by 
$x(T)$ we understand a block $B$ with domain $T$ such that for every $t\in T$, 
$B(t)=x(t)$. In a slight abuse of notation, we will not distinguish between two 
blocks if their domains differ only by translation, i.e. for any $g$ we treat 
$x(T)$ and $x(Tg)$ as the same block.

If $X$ is a symbolic dynamical system over $\Lambda$, then we say that a block 
$B$ with domain $T$ \emph{occurs in $X$}, if $x(T)=B$ for some $x\in X$. 
Finally, if we denote by $\CN_T(X)$ the number of blocks with domain $T$ which 
occur in $X$, we can calculate the \emph{topological entropy of $X$} as the 
limit
\[h(X)=\lim\limits_{n\to\infty}\frac{1}{\abs{F_n}}\log\abs{\CN_{F_n}(X)}\]
(where $\log$ means logarithm with base $2$).
It is known that this limit always exists and does not depend on the choice of 
the F\o lner sequence (see Theorem 6.1 in \cite{LW}). In fact, we will make use of the following consequence of the cited theorem:
\begin{prop}	\label{prop:NT_estimate}
	For any $\eps$ there exists an $N$ and $\delta$ such that if $T$ is an 
	$(F_n,\delta)$-invariant set for some $n>N$, then 
	$\CN_T(X)>2^{(h(X)-\eps)\abs{T}}$.
	\end{prop}
We will also assume, that for every finite set $D$ we have $D\subset F_n$ for sufficiently large $n$.
	
\section{Tilings}

We briefly recall the notions and most important results concerning tilings of 
amenable groups; for details we refer the reader to \cite{DHZ}. 
\begin{defn}
A \emph{tiling} of an amenable group $G$ is a collection $\CT$ 
of finite subsets of $G$, such that:
\begin{itemize}
	\item $T_1\cap T_2=\emptyset$ whenever $T_1,T_2$ are two different elements 
	of $\CT$.
	\item $\bigcup_{T\in \CT}T=G$
	\item There exists a finite family $\CS=(S_1,\ldots,S_k)$ of finite subsets 
	of $G$, such that every $T$ can be uniquely represented in the form $S_jg$ 
	for some $j\in\set{1,2,\ldots,k}$ and $g\in G$.
\end{itemize}
The elements of $\CT$ are referred to as \emph{tiles}, and the elements of 
$\CS$ are referred to as \emph{shapes}. Also, we can define a mapping $\sigma:\CT\to 
\CS$ such that $S_j=\sigma(T)$ if and only if $T=S_jg$ for some $g\in G$.
\end{defn}
Every tiling $\CT$ with shapes $S_1,\ldots,S_k$ induces an element $x_{\CT}$ of 
$\set{0,1,\ldots,k}^G$, defined by requesting that if $S_jg\in\CT$ for some 
$j$, then $x_{\CT}(g)=j$ (the properties constituting the definition of a 
tiling ensure that such a $j$ is unique), and if no such $j$ exists, then 
$x_{\CT}(g)=0$. This in turn allows us to associate with $\CT$ a symbolic 
dynamical system $X_{\CT}$ defined as the orbit closure of $x_{\CT}$ under the 
shift action. Note that by reversing the procedure which defines $x_{\CT}$, we 
can obtain from every element of $X_{\CT}$ another tiling of $G$ which uses the 
same collection of shapes as $\CT$.

In view of this discussion, it is natural to define the action of $G$ directly on tilings of $G$ by putting:
\[
g\CT=\set{Tg^{-1}:T\in \CT}
\]
Clearly, $g(x_\CT)(h)=j$ if and only 
if $x_\CT(hg)=j$, so $S_jh$ is a tile of $g\CT$ if and only if $S_jhg$ is a tile 
of $\CT$ and this definition is consistent with the shift action of $G$ on $X_\CT$.
 
We will need the following result which is an immediate consequence of theorem 
5.2 of \cite{DHZ}:
\begin{thm}\label{thm:tiling}
	If $G$ is a countable amenable group and $K\subset G$ is a finite set, then 
	for 
	every $\eps>0$ there exists a tiling $\CT$ of $G$ such that the shapes of 
	$\CT$ are $(K,\eps)$-invariant sets, and the system $X_{\CT}$ has entropy 
	zero.
\end{thm}
\section{The main result}
We begin by introducing a property analogous to the corner/centre gluing conditions used in the $\mathbb{Z}^d$ context.
\begin{defn}	\label{def:gluing}
We say that a symbolic dynamical system $(X,G)$ has the \emph{gluing 
property} if there exists a finite set $D$ (containing the neutral element) 
such that for any finite subsets $T_1$ and $T_2$ of $G$, such that $T_2\cap 
DT_1=\emptyset$, and any two blocks $A$ 
and $B$, with domains respectively $T_1$ and $T_2$, there exists an 
$x\in X$ such that $x(T_1)=A$ and $x(T_2)=B$. The set $D$ will be 
referred to as the \emph{gluing distance}.
\end{defn}

\begin{thm}\label{thm:main}
If the symbolic dynamical system $(X,G)$ with topological entropy greater than 
$\log k$ has the gluing property, then there exists a symbolic 
extension $(\tilde{X},G)$ of $X$, having the same topological entropy as $X$, 
and such that 
$(\tilde{X},G)$ factors onto the 
full shift over $k$ symbols.
\end{thm}
\begin{proof}
Let $l$ be the number of symbols in the alphabet of $X$, let $\gamma$ be a 
number such that $1<\gamma<\frac{h(X)}{\log k}$, and let $D$ be the gluing 
distance. There exists a tiling 
$\CT_0$ 
of $G$, such that $X_{\CT_0}$ has topological entropy zero, and for every shape 
$S$ of $\CT_0$:
\begin{enumerate}
\item $\abs{S\setminus S_D}<((\gamma-1)\log_l k)\abs{S}$.
\item\label{manyblocks} The number of blocks with domain $S$ occurring in $X$ 
is greater than $k^{\gamma \abs{S}}$.
\end{enumerate}
Indeed, using proposition \ref{prop:NT_estimate} for $\eps< h(X)-\gamma\log k$ we obtain $N$ and $\delta$ such that (2) holds for blocks $S$ which are $(F_n,\delta)$-invariant, where $n\geq N$. Then, by theorem \ref{thm:tiling} we get a $(F_n,\eps)$-invariant tiling, where $n\geq N$ and $0<\eps<\delta$ are appropriately chosen to guarantee (1) (in particular, we demand that $D \subset F_n$).

Combining the properties (1) and (2), we can estimate from below the number of 
blocks 
with domain $S_D$ occurring in $X$: If we denote by $\CN_1$ the number of 
blocks with domain $S$, and by $\CN_2$ the number of blocks with domain $S_D$, 
then we have
\[k^{\gamma \abs{S}}<\CN_1<\CN_2 l^{\abs{S\setminus 
S_D}}<\CN_2l^{((\gamma-1)\log_l k)\abs{S}}=\CN_2k^{(\gamma-1)\abs{S}}, \]
hence
\[\CN_2>k^{\abs{S}}.\]
It follows that for every shape $S$ of $\CT_0$ we can construct a mapping 
$\phi_S$ from the collection of all 
blocks with domain $S_D$ occurring in $X$ onto $\{1,...,k\}^S$. 

We can now create the symbolic 
dynamical system $(\tilde{X},G)$ as the product of $(X,G)$ and $(X_{\CT_0},G)$. 
This is obviously an extension of $X$, and since the entropy of the product is 
equal to the sum of entropies of both systems, $\tilde{X}$ has entropy equal to 
$X$.
Every element $\tilde{x}$ of $\tilde{X}$ consists of a pair $(x,\CT)$ where 
$x\in X$ and $\CT$ is a tiling of $G$ using the same shapes as $\CT_0$. We can 
now define a map $\phi:\tilde{X}\to\{1,...,k\}^G$ as follows: for every 
$\tilde{x}=(x,\CT)$ let $y=\phi(\tilde{x})$ be defined by requiring that for 
every $T\in \CT$ we have 
$y(T)=\phi_{\sigma(T)}(x(T_D))$. Since $y(T)$ depends only on $x(T)$, this map is 
continuous. It is also easy to verify that it commutes with the shift: If 
$\tilde{x}=(x,\CT)$ is an element of $\tilde{X}$, then $g\tilde{x}=(gx,g\CT)$. 
For every tile $Tg^{-1}$ of 
$g\CT$ we have 
\[(\phi(g\tilde{x}))(Tg^{-1})=\phi_{\sigma(Tg^{-1})}\big(gx((Tg^{-1})_D)\big)=\phi_{\sigma(T)}(x(T_D))=
\phi(\tilde{x})(T),\]
therefore, since $T$ was arbitrary, for every $h\in G$ we have
\[(\phi(g\tilde{x}))(h)=\phi(\tilde{x})(hg).\]

It remains to verify that $\phi$ is onto. Let $y$ be any element of $\{1,...,k\}^G$ 
and let $T_1,T_2,\ldots$ be an enumeration of the tiles of $\CT_0$ and let $B_i=y(T_i)$. There 
exists some 
$x_1\in X$ such that 
$\phi_{\sigma(T_1)}(x_1( (T_1)_D))=B_1$, and thus for $\tilde{x}_1=(x_1,\CT_0)$ we have 
$\phi(\tilde{x}_1)(T_1)=B_1$. Now, suppose that for some $j$ we have already 
found an $\tilde{x}_j\in \tilde{X}$ such that for every $i\leq j$ we have 
$\phi(\tilde{x}_j)(T_i)=B_i$ (so far we know this is possible for $j=1$). There 
exists some 
$x'_{j+1}\in X$ such that 
$\phi_{\sigma(T_{j+1})}(x'_{j+1}( (T_{j+1})_D))=B_{j+1}$. Now, the sets $D(T_{j+1})_D$ 
and $T_1\cup T_2\cup\ldots\cup T_j$ are disjoint, so the gluing property means 
there exists some $x_{j+1}$ such that $x_{j+1}(T_i)=x_j(T_i)$ for 
$i=1,\ldots,j$, and $x_{j+1}((T_{j+1})_D)=x'_{j+1}((T_{j+1})_D)$. If we now set 
$\tilde{x}_{j+1}=(x_{j+1},\CT_0)$, we will have 
$\phi(\tilde{x}_{j+1})(T_i)=B_i$ for every $i\leq j+1$. By the principle of 
mathematical induction we obtain that for every $j$ there exists an 
$\tilde{x}_j\in \tilde{X}$ such that for every $i\leq j$ we have 
$\phi(\tilde{x}_j)(T_i)=B_i$. Since $\tilde{X}$ is compact, there exists a 
convergent subsequence of $(\tilde{x}_j)$ converging to some $\tilde{x}\in 
\tilde{X}$. Hence for every $i$ there exists some $j\geq i$ such that 
$\tilde{x}(T_i)=\tilde{x}_j(T_i)$, but then 
$\phi(\tilde{x})(T_i)=\phi(\tilde{x}_j)(T_i)=B_i$. Since $i$ was arbitrary, 
$\phi(\tilde{x})=y$, and thus $\phi$ is onto.
\end{proof}
\section*{Acknowledgements}
The research is funded by NCN grant 2013/08/A/ST1/00275.


\end{document}